\pdfoutput=1
\documentclass[11pt]{article}
\usepackage{amssymb}
\usepackage[all,cmtip]{xy}
\usepackage[cmtip]{xypic}
\usepackage[width=17cm,top = 3.5cm,bottom=3.5cm]{geometry}
\usepackage{amsthm,amsfonts,amssymb,amsmath,amscd} 
\usepackage{aliascnt} 
\usepackage{mathrsfs} 
\usepackage{xargs,ifthen} 

\usepackage{comment}
\usepackage{color}

\usepackage{hyperref}
\definecolor{tocolor}{rgb}{.1,.1,.1}
\definecolor{urlcolor}{rgb}{.2,.2,.6}
\definecolor{linkcolor}{rgb}{.1,.1,.5}
\definecolor{citecolor}{rgb}{.4,.2,.1}
\hypersetup{backref=true, colorlinks=true, urlcolor=urlcolor, linkcolor=linkcolor, citecolor=citecolor}

\newcommandx{\thdef}[2]{
	\newaliascnt{#1}{theorem}  
	\newtheorem{#1}[#1]{#2}
	\aliascntresetthe{#1}  
	\newtheorem*{#1*}{#2}
	\expandafter\newcommand\expandafter{\csname #1autorefname\endcsname}{#2}
}

\makeatletter
\newtheorem*{rep@theorem}{\rep@title}
\newcommand{\newreptheorem}[2]{%
\newenvironment{rep#1}[1]{%
 \def\rep@title{#2 \ref{##1}}%
 \begin{rep@theorem}}%
 {\end{rep@theorem}}}
\makeatother

\newtheorem{theorem}{Theorem}[section]
\newreptheorem{theorem}{Theorem}

\thdef{lemma}{Lemma}
\thdef{corollary}{Corollary}
\thdef{conjecture}{Conjecture}
\newreptheorem{conjecture}{Conjecture}
\thdef{proposition}{Proposition}

\theoremstyle{definition}
\thdef{definition}{Definition}
\thdef{notation}{Notation}

\theoremstyle{remark}
\thdef{remark}{Remark}

\theoremstyle{remark}
\thdef{ex}{Example}

{\hfill $\blacksquare$\end{ex}}

\newcommand{\spc}[1]{\mathsf{#1}} 
\newcommand{\shf}[1]{\mathcal{#1}} 

\newcommand{\RR}{\mathbb{R}}
\newcommand{\CC}{\mathbb{C}}
\newcommand{\ZZ}{\mathbb{Z}}

\newcommand{\rbrac}[1]{\left(#1\right)} 

\newcommandx{\fn}[2][2=]{#1\ifthenelse{\equal{#2}{}}{}{\!\rbrac{{#2}}}} 
\newcommandx{\id}[2][2=]{\fn{{\rm id}_{#1}}[#2]} 

\newcommand{\ext}[2][\bullet]{\spc{\Lambda}^{#1}{#2}} 
\newcommandx{\End}[2][1=]{\fn{\spc{End}_{#1}}[#2]} 
\newcommandx{\Hom}[2][1=]{\fn{\spc{Hom}_{#1}}[#2]} 
\newcommandx{\Aut}[2][1=]{\fn{\spc{Aut}_{#1}}[#2]} 
\newcommandx{\image}[1]{\fn{\spc{img}}[#1]} 
\renewcommandx{\ker}[1]{\fn{\spc{ker}}[#1]} 
\newcommandx{\rank}[1]{\fn{\mathrm{rank}}[#1]} 

\newcommandx{\ann}[1]{\fn{\spc{ann}}[#1]} 

\newcommandx{\hlgy}[3][1=\bullet,3=]{\spc{H}_{#1}^{#3}\!\rbrac{{#2}}} 
\newcommandx{\cohlgy}[3][1=\bullet,3=]{\spc{H}^{#1}_{#3}\!\rbrac{{#2}}} 
\newcommandx{\chow}[3][1=\bullet,3=]{\spc{A}^{#1}_{#3}\!\rbrac{{#2}}} 

\newcommandx{\Ext}[3][1=\bullet,3=]{\fn{\spc{Ext}^{#1}_{#3}}[{#2}]} 
\newcommandx{\Tor}[3][1=\bullet,3=]{\fn{\spc{Tor}^{#1}_{#3}}[{#2}]} 

\newcommandx{\Pic}[1]{\fn{\spc{Pic}}[{#1}]} 
\newcommandx{\chernalg}[2][1=\bullet]{\fn{\spc{Chern}^{#1}}[{#2}]} 
\newcommandx{\chern}[2][1=]{\fn{c_{#1}}[#2]} 
\newcommandx{\ch}[2][1=]{\fn{\mathrm{ch}_{#1}}[{#2}]} 

\newcommandx{\sKer}[2][1=]{ \fn{ \shf{K}er_{#1}}[{#2}] } 
\newcommandx{\sHom}[2][1=]{ \fn{ \shf{H}om_{#1}}[{#2}] } 
\newcommandx{\sEnd}[2][1=]{ \fn{ \shf{E}nd_{#1}}[{#2}] } 
\newcommandx{\sExt}[3][1=\bullet,3=]{\fn{\shf{E}xt^{#1}_{#3}}[{#2}]} 
\newcommandx{\sTor}[3][1=\bullet,3=]{\fn{\shf{T}or^{#1}_{#3}}[{#2}]} 

\newcommandx{\forms}[2][1=\bullet]{\Omega^{#1}_{#2}} 
\newcommandx{\can}[1][1=]{\omega_{#1}} 
\newcommandx{\acan}[1][1=]{\omega_{#1}^{-1}} 
\newcommandx{\tshf}[1]{\shf{T}_{#1}} 
\newcommandx{\mvect}[2][1=\bullet]{ \ext[#1]{\tshf{#2}} }
\newcommandx{\der}[2][1=\bullet]{\mathscr{X}^{#1}_{#2}} 
\newcommandx{\sJet}[3][1=,2=]{\shf{J}^{#1}_{#2}#3} 

\newcommandx{\tb}[2][1=]{\spc{T}_{\!#1}{#2}} 
\newcommandx{\ctb}[2][1=]{\spc{T}_{\!#1}^*{#2}} 
\newcommandx{\lie}[2][2=]{\fn{\mathscr{L}_{#1}}[#2]} 
\newcommandx{\hook}[2][2=]{\fn{i_{#1}}[#2]} 

\newcommand{\del}{\partial}

\makeatletter
\newcommand{\thickbar}{\mathpalette\@thickbar}
\newcommand{\@thickbar}[2]{{#1\mkern1.5mu\vbox{
  \sbox\z@{$#1\mkern-1mu#2\mkern-1mu$}%
  \sbox\tw@{$#1\overline{#2}$}%
  \dimen@=\dimexpr\ht\tw@-\ht\z@-.6\p@\relax
  \hrule\@height.4\p@ 
  \vskip1\p@
  \hrule\@height.4\p@ 
  \vskip\dimen@
  \box\z@}\mkern1.5mu}
}
\makeatother

\numberwithin{equation}{section}

\newtheoremstyle{parag}
  {\topsep}   
  {\topsep}   
  {}  
  {}       
  {\bfseries} 
  {.}         
  { } 
  {}          
\theoremstyle{parag}

\makeatletter
\def\@cite#1#2{{\normalfont[{#1\if@tempswa , #2\fi}]}}
\makeatother

\newcommand{\J}{\mathbb J}

\newcommand{\delbar}{\overline\partial}
\newcommand{\tensor}{\otimes}

\newcommand{\TM}{\mathbb{T}M}
\newcommand{\T}{\mathbb{T}}

\newcommand{\U}{\mathcal{U}}

\newcommand{\gcm}{generalized complex manifold}
\newcommand{\gcs}{generalized complex structure}
\newcommand{\into}{\to}
\newcommand{\wrt}{with respect to}

\newcommand{\ol}{\overline}

\newcommand{\mc}[1]{\text{$\mathcal{#1}$}}

\newcommand{\C}{\mathbb{C}}

\renewcommand{\Im}{\mathrm{Im}}

\newcommand{\gcss}{generalized complex structures}

\newcommand{\gcy}{generalized Calabi--Yau}
\newcommand{\gcys}{generalized Calabi--Yau structure}
\newcommand{\gcyss}{generalized Calabi--Yau structures}

\newcommand{\gf}{\text{$\varphi$}}

\newcommand{\R}{\text{${\mathbb R}$}}
\renewcommand{\iff}{if and only if}

\begin{document}

\title{\vspace{-4em} \huge Type one generalized Calabi--Yaus}
\date{}

\author{
Michael Bailey\thanks{Utrecht University; {\tt m.a.bailey@uu.nl}}
\and
Gil R. Cavalcanti\thanks{Utrecht University; {\tt g.r.cavalcanti@uu.nl}}
\and
Marco Gualtieri \thanks{University of Toronto; {\tt mgualt@math.toronto.edu}}
}
\maketitle

\renewcommand{\abstractname}{\vspace{-\baselineskip}}

\abstract{We study type one generalized complex and generalized Calabi--Yau manifolds. We introduce a  cohomology class that obstructs the existence of a globally defined, closed 2-form which agrees with the symplectic form on the leaves of the generalized complex structure, the {\it twisting class}.  We prove that in a compact, type one,  $4n$-dimensional generalized complex manifold the Euler characteristic must be even and equal to the signature modulo four. The generalized Calabi--Yau condition places much stronger constrains: a compact type one generalized Calabi--Yau fibers over the 2-torus and if the structure has one compact leaf, then this fibration can be chosen to be the fibration by the symplectic leaves of the \gcs. If the twisting class vanishes, one can always deform the structure so that it has a compact leaf. Finally we prove that every symplectic fibration over the 2-torus admits a type one \gcys.
}

\renewcommand\contentsname{\vspace{-\baselineskip}}

\section*{Introduction}

Generalized complex structures, introduced by Hitchin in 2003 \cite{MR2013140} and Gualtieri \cite{MR2811595}, are a simultaneous generalization of complex and symplectic structures. At each point $p$ in the manifold $M$, a \gcs\ is equivalent to a symplectic subspace of $T_pM$ together with a transverse complex structure. The \emph{type} of the \gcs\ at the point in question is the complex dimension of this transverse complex space.  Thus, symplectic structures have type 0 everywhere and complex structures have maximal type everywhere. In general, \gcss\ determine singular symplectic foliations with a transverse complex structure and, if the type is constant, they determine ordinary symplectic foliations.

The type of a generalized complex structure on $M$ is an integer-valued lower semi-continuous function with locally constant parity.  
In the generic even case, \gcss\ may be viewed as symplectic structures with singularities along loci where the type jumps from $0$ to $2$ or higher; when these singular loci are required to satisfy a transversality condition, we have {\it stable \gcss}, which were studied by Cavalcanti and Gualtieri in \cite{Cavalcanti:2015uua}. If, instead, we are in the case of odd type, a generic generalized complex structure would be of type 1 almost everywhere; very little is currently known about type 1 structures.

Besides being the odd analogue of symplectic structures, type 1 structures are closely related to stable (even) \gcss\ in a more direct way: the singular locus of the symplectic form in the stable case inherits a type 1 \emph{\gcys}.  So, the study of stable structures naturally leads to the study of type 1 \gcyss.  The objective of this paper is to determine the basic differential topological properties of these structures.

We prove that a compact, connected, type 1 \gcy\ manifold, $M$, has a rather restricted topology: $M$ must be a fiber bundle over the 2-torus. Further, if the manifold has at least one compact symplectic leaf, then all leaves are compact symplectic manifolds, and $M$ fibers over its symplectic leaf space, which is $T^2$. We also prove a partial converse to this statement: every compact symplectic fibration over $T^2$ admits a \gcys\ for which the symplectic leaves are the fibers of the fibration.  As a special case, we obtain a correspondence in four dimensions: a compact four-manifold admits a type 1 \gcys\ \iff\ it is an oriented fibration over $T^2$. These results may be viewed as the generalized complex analogues of the results obtained by Guillemin, Miranda and Pires for codimension one Poisson structures \cite{MR2861781}.

\section{Topology of type one generalized complex structures}\label{sec:gcss}

\subsection{The twisting class of a generalized complex structure }

\begin{definition}
A  {\it \gcs} on a manifold $M$ with closed 3-form $H$ is a complex structure $\J$ on $\TM = TM \oplus T^*M$  compatible with the natural pairing of vectors on forms and integrable \wrt\ the Courant bracket twisted by  $H$.
\end{definition}

Alternatively, $\J$ is fully determined by its $+i$-eigenspace $L \subset \T_\C M$, a maximal isotropic, involutive  sub-bundle satisfying $L \cap \overline{L} =\{0\}$. Furthermore, $\J$ can be fully described using differential forms:

\begin{definition}
A {\it \gcs} on a manifold with closed 3-form $(M^{2n},H)$, is a complex line bundle  $K \subset \wedge^{\bullet}T^*_\C M$ such that
\begin{enumerate}
\item $K$ is generated pointwise by a form $\rho$ of the following algebraic type
$$\rho = e^{B+ i \omega} \wedge \Omega,$$
where $\Omega$ is a decomposable form and $B$ and $\omega$ are real two-forms;
\item Pointwise, for the generator above,
$$\Omega \wedge \bar{\Omega} \wedge \omega^{n-k}\neq 0,$$
where $k$ is the degree of $\Omega$;
\item For every nonvanishing local section $\rho$ there is $X+\xi \in \Gamma(\overline{L})$ such that
$$d^H\rho := d\rho + H\wedge \rho= \iota_X \rho + \xi \wedge \rho.$$
\end{enumerate}
The degree of the form $\Omega$ at a point $p$ is the {\it type} of the \gcs\ at $p$, the line bundle $K$ is the {\it canonical bundle} and $X+\xi$ is the {\it modular field} corresponding to the trivialization $\rho$.
\end{definition}

\begin{definition}
A {\it generalized Calabi--Yau structure} on $(M,H)$ is a \gcs\ determined by a nowhere vanishing $d^H$-closed form.
\end{definition}

Examples of generalized complex manifolds include symplectic manifolds, $(M,\omega)$, where $K$ is the line generated by $e^{i\omega}$; complex manifolds, where $K = \wedge^{n,0}T^*M$ is the usual canonical bundle; and holomorphic Poisson manifolds $(M,I,\pi)$ where $K = e^{\pi}\cdot \wedge^{n,0}T^*M$ and $\pi$ acts on forms by interior product. 

From the definition we see that, pointwise, the subspace $\mc{D}$ annihilating $\Omega\wedge\bar{\Omega}$ is the complexification of a real subspace of $TM$ and $\omega$ is a symplectic structure on $\mc{D}$. If the type is constant, $\mc{D}$ is an integrable distribution and $\omega$ is a symplectic structure on the fibers. In general, $\mc{D}$ is an integrable singular distribution and $\omega$ gives a symplectic structure to its leaves.

Any \gcs\ $\J$ on $M^{2n}$ decomposes the space of forms into its $ik$-eigenspaces: $U^k$. These are nontrivial for all the integers $k$ between $-n$ and $n$ with $U^n= K$ and $U^{n-k} = \wedge^k \overline{L} \cdot K$. Further the operator $d^H$ also decomposes as a sum $d^H = \del + \delbar$ with
$$\del:\U^k \into \U^{k+1}\quad\mbox{and}\qquad \delbar:\U^k\into \U^{k-1}.$$

Since $L$ is involutive, it is a Lie algebroid over $M$ and using the nondegenerate pairing to identify $\bar{L} = L^*$, $\Gamma(\wedge^\bullet\overline{L})$ becomes a differential graded Lie algebra (DGLA) with $d_L$, the Lie algebroid differential from $L$, and the Schouten--Nijenhuis extension of the Courant bracket as a bracket. Further the space of forms with the operator $\delbar$ is a differential module for this DGLA, i.e., for all $\rho \in \Omega^{\bullet}(M;\C)$ and $\alpha \in \Gamma(\wedge^{\bullet}\overline{L})$ we have 
\begin{equation}\label{eq:differential module}
\{\delbar,\alpha\} \rho = (d_L\alpha) \rho,
\end{equation}
where $\{\cdot,\cdot\}$ is the graded commutator of linear differential operators on forms.

It follows directly from \eqref{eq:differential module} that any (local) modular field is $d_L$-closed. If the canonical bundle is trivial, a global nowhere vanishing section $\rho \in \Gamma(K)$ gives rise to a global modular field $v$. Another nonvanishing section is just a multiple of  $\rho$, say  $g \rho$, where $g:M\into \C^*$, and hence the modular field of this new section is  given by $ v + d_L \log g$ and, if $K$ is trivial, the cohomology class of the  modular field in the complex
$$\xymatrix{
0\ar[r] & \Gamma(M;\C^*) \ar[r]^{d_L \circ \log} & \Gamma(M;\overline{L}) \ar[r]^{d_L}& \Gamma(M;\wedge^2\overline{L}) \ar[r]^{d_L}& \cdots
}$$
is well defined, independent of choice of trivialization and is called the {\it modular class} of the \gcs. If this class is trivial, one can produce a nowhere vanishing $d^H$-closed section of $K$ and the \gcs\ is in fact \gcy.

For \gcss\ of constant type, there is yet another cohomology class determined by the structure that we describe next. Since the type is constant, locally the \gcs\ is determined by a $d^H$-closed form  $e^{B+i\omega}\wedge \Omega$ \cite{MR2811595}. Notice that the form $\Omega$ is not unique: only the line it determines is intrinsic. Similarly, $B+i\omega$ is not unique as it can be changed by any 2-form which wedges zero with $\Omega$ and will still determine the same spinor. So to better understand these quantities we introduce two differential complexes.

Firtly, we let $\mc{I}^\bullet$ be the algebraic ideal generated (pointwise) by the 1-forms which make up $\Omega$. This ideal is independent of the trivialization chosen and is in fact a differential subalgebra of $\Omega^\bullet(M;\CC)$. We can therefore form the quotient which is again a differential complex, so we have the following short exact sequence of differential complexes
\begin{equation}\label{eq:short sequence}
0\into \mc{I}^l \into \Omega^l(M) \into \Omega^{l}(M)/\mc{I}^l \into 0.
\end{equation}
It is worth noticing that this sequence can also be made into a (periodic) sequence of differential complexes with differential $d^H$:
\begin{equation}\label{eq:short sequence H}
\xymatrix@R=4ex@C=3ex{
&\ar[d]^-{d^H}&\ar[d]^-{d^H}&\ar[d]^-{d^H}&\\
0\ar[r] &\mc{I}^{ev} \ar[r]\ar[d]^-{d^H}& \Omega^{ev}(M) \ar[r]\ar[d]^-{d^H} &\Omega^{ev}(M)/\mc{I}^{ev} \ar[r]\ar[d]^-{d^H}& 0\\
0\ar[r] &\mc{I}^{od} \ar[r]\ar[d]^-{d^H}& \Omega^{od}(M) \ar[r]\ar[d]^-{d^H} &\Omega^{od}(M)/\mc{I}^{od} \ar[r]\ar[d]^-{d^H}& 0\\
&&&&
}\end{equation}

From the previous discussion we see that, while $B+ i\omega$ is not well defined as a form, it does give rise to a well defined element $[B+i\omega] \in  \Omega^{2}(M)/\mc{I}^2$ and the particular choice of  forms $B+i\omega$ making  up the spinor  is nothing but an arbitrary choice of pre-image of the intrinsic element in $\Omega^{2}(M)/\mc{I}^2$. Further, the generalized Calabi--Yau condition, $d^H (e^{B+i\omega}\wedge \Omega) =0$, implies that $[e^{B+i\omega}] \in \Omega^{ev}(M)/\mc{I}^{ev}$ is closed for the corresponding $d^H$ operator and $[H + d(B+i\omega)] \in \Omega^{3}(M)/\mc{I}^{3}$ is closed for the corresponding $d$ operator. 

\begin{proposition}
The following are equivalent
\begin{itemize}
\item[a)] The element $[e^{B+i\omega}] \in \Omega^{ev}(M;\C)/\mc{I}^{ev}$ can be represented by a $d^H$-closed form
\item[b)] $\delta [e^{B+i\omega}] =0$, where $\delta$ is the connecting homomorphism in cohomology for the sequence \eqref{eq:short sequence H}
\item[c)]  $H + d(B+ i\omega)$ represents the trivial cohomology class in $\mc{I}^3$.
\end{itemize}  
\end{proposition}
\begin{proof}~
\begin{itemize}
\item [a) $\Rightarrow$ b)] By construction of the connecting homomorphism, if $[e^{B+i\omega}]$ is represented be a form $\alpha$, then $d^H\alpha \in \mc{I}^{od}$ is a closed element which represents the  class $\delta[e^{B+i\omega}]$, therefore if $[e^{B+i\omega}]$ can be represented by a  $d^H$-closed form, then $\delta[e^{B+i\omega}] =0$.
\item [b) $\Rightarrow$ c)] If  $\delta [e^{B+i\omega}] =0$, then there is $\beta \in \mc{I}^{ev}$ such that
\begin{equation}\label{eq:b=>c}
d^He^{B+i\omega} = d^H\beta.
\end{equation}
Since $\mc{I}^0 = \{0\}$, the degree 3 part of \eqref{eq:b=>c} is
$$H+ d(B+i\omega) = d\beta_2,$$
where $\beta_2 \in \mc{I}^2$ is the degree 2 component of $\beta$. In particular we get that $H+ d(B+i\omega) $ represents the trivial cohomology class in $H_{\mc{I}}^2$. 
\item [c) $\Rightarrow$ a)] If $H +d(B+i\omega)$ represents an exact class in $H_{\mc{I}}^3$, then there exists $\beta \in \mc{I}^2$ such that $H +d(B+i\omega) + d\beta=0$. Notice that $B+ i\omega+ \beta$ represents the same class as $B +i\omega$ in $\Omega^{2}(M)/\mc{I}^{2}$, hence $[e^{B+i\omega}]$ is represented by $e^{B+i\omega+\beta}$ and for this form we have
$$d^He^{B+i\omega+\beta} = (H + d(B+i\omega+ \beta))e^{B+i\omega+ \beta} = 0.$$
\end{itemize}
\end{proof}

\begin{definition}
The class $[H + d(B+i\omega)] \in H^3_{\mc{I}}$ is the {\it twisting class} of the generalized complex structure.
\end{definition}
This class measures the failure of existence of a global $d^H$-closed form $e^{B+i\omega}$ for which the \gcs\ is locally given by closed form $e^{B+i\omega}\Omega$. Notice that if $H$ is nontrivial in de Rham cohomology the twisting class is automatically nontrivial.

Summarising, we have the following obstructions associated to a \gcs:

\begin{proposition}~
\begin{enumerate}
\item The canonical bundle, $K$, of a \gcm\ has a nowhere vanishing global section \iff\ $c_1(K) =0$;
\item A \gcs\ with $c_1(K) =0$ is \gcy\ \iff\ the modular class vanishes;
\item A \gcs\ of constant type is determined locally by a closed form $e^{B+i\omega} \wedge \Omega$ with $e^{B+i\omega}$ globally defined and satisfying  $d^H e^{B+i\omega} =0$ \iff\ the twisting class vanishes.
\end{enumerate}
\end{proposition}

\subsection{Topological constraints}

Having a type 1 generalized complex structure already places constrains on the topology of the manifold. The most basic one is related to the underlying linear algebra of a type 1 structure (i.e. integrability is not necessary). Indeed an almost generalized  complex structure of type 1 induces an oriented distribution of co-dimension 2 which can always be complemented to an oriented rank 2 distribution. The existence of such distributions in itself already places the first topological restriction.

\begin{theorem}[Atiyah \cite{MR0263102}]
If a compact oriented manifold $M^{4n}$ admits an oriented 2-plane field then the Euler characteristic of $M$ is even and congruent to the signature of $M$ modulo four.
\end{theorem}
So one immediately sees that, for example, $\C P^{2n}$ does not admit type 1 \gcss\ for any $n>0$.

Notice however that once a co-dimension 2 distribution has been found on a manifold, one can always deform it into a  $C^0$ distribution which is integrable with smooth leaves \cite{MR0370619}. Hence the difference between the existence of a codimension two distribution and a $C^{\infty}$ integrable one  is rather subtle.

We now describe several topological properties of  a type 1 generalized Calabi-Yau, and indicate how these are affected by the vanishing of the twisting class introduced above.  	   

\begin{definition}
A \gcm\  is {\it proper} if its symplectic leaves are proper submanifolds.
\end{definition}

\begin{theorem}\label{theo:topology1} Let $(M^{2n},H)$ be a compact, connected type one \gcy\ manifold. Then all of the following hold:
\begin{enumerate}
\item There is a surjective submersion $\pi: M \into T^2$.
\item If $M$ has a compact leaf, then $M$ is proper and $\pi$ can be chosen so that the components of the fibers of $\pi$  are the symplectic leaves of the underlying Poisson structure.
\item If the twisting class vanishes, then the structure can be deformed into a proper one, $M$ admits a symplectic structure for which $\pi:M \into T^2$ is a symplectic fibration, and there are classes $a,b \in H^1(M)$ and $c \in H^2(M)$ such that $abc^{n-1} \neq 0$. In particular $b_i(M)\geq 2 $ for $0<i<2n$.
\end{enumerate}
\end{theorem}
\begin{proof}
Throughout the proof we let $\rho = e^{B+i\omega}\wedge \Omega$ be a $d^H$-closed trivialization of the canonical bundle of $M$.

{\it 1.}  Let $\Omega_R$ and $\Omega_I$ be the real and imaginary parts of $\Omega$. First we show that $[\Omega_R]$ and $[\Omega_I]$ are linearly independent classes in $H^1(M)$. If there was a nontrivial linear combination, say, $\lambda_R[\Omega_R] + \lambda_I[\Omega_I] =0$,  we could define a map $\pi_1:M\into \R$ by
$$\pi_1(p)= \int_{p_0}^p \lambda_R\Omega_R + \lambda_I\Omega_I,$$
where the integral, taken over any path connecting the reference point $p_0$ to $p$, is well defined  because the integrand is an exact form.  Finally, if, say, $\lambda _R \neq 0$, then nondegeneracy implies that $\omega^{n-1}\wedge d\pi_1 \wedge \Omega_I \neq 0$, showing that $d\pi_1$ is nowhere zero and hence $\pi_1$ is a submersion of a compact manifold in $\R$, which is a contradiction.

To construct $\pi$ observe that since nondegeneracy is an open condition, there is a closed form $\Omega'\in \Omega^1(M;\C)$ near $\Omega$, whose real and imaginary parts represent linearly independent rational cohomology classes  and such that
\begin{equation}\label{eq:nondegeneracy condition 2}
\omega^{n-1}\wedge\Omega'\wedge \overline{\Omega'} \neq 0.
\end{equation}
Linear independence of the real and imaginary parts of $\Omega'$ implies that the following map is  a submersion
$$\pi:M\into \C/\Lambda; \qquad \pi(p) =\int_{p_0}^p \Omega'.$$
where $\Lambda = [\Omega'](H_1(M;\ZZ))$ is a co-compact lattice,
$p_0$ is a fixed reference point and the integral is independent of the path connecting $p_0$ to $p$.

{\it 2.} One way to prove this statement is to use the following result:

\begin{theorem}[Brunella \cite{MR1448725}]
Let $\mc{F}$ be a transversally holomorphic foliation of complex co-dimension one on a compact connected manifold $M$. Assume that there exists a compact leaf $L \in \mc{F}$ with finite holonomy. Then $\mc{F}$ is a Seifert fibration.
\end{theorem}

In the present situation, in a tubular neighbourhood $U$ of a compact leaf, $L$, the 1-form $\Omega$ is exact, since it is closed and its restriction to $L$ vanishes. Say $\Omega = dz$ for some complex function $z$ defined on $U$. It follows that $z$ is constant along the leaves and since $\Omega\wedge\bar\Omega \neq 0$, $z:U\to \C$ is a submersion, that is $z$ parametrizes the local leaf space and $L$ has trivial holonomy. Due to Brunella's theorem, $M$ is a Seifert fibration, in particular every leaf is compact. The argument we just used implies that every leaf has zero holonomy and hence the foliation is in fact a regular fibration and the form $\Omega$ is basic thus passes to the (smooth) leaf space, giving it the structure of an elliptic curve.

The proof of Brunella's result is rather analytic. Next we present an alternative more geometric proof  of our result.

Let $X$ be a complex vector field for which  $\Omega(X)=1$ and $\ol{\Omega}(X)=0$. Then the real and imaginary parts of $X$, $X_R$ and $X_I$, are pointwise linearly independent, preserve $\Omega$ and $\ol\Omega$ and hence preserve the foliation determined by $\Omega\wedge \ol\Omega$.

Let $F$ be a compact leaf and define a map $\gf_F:F \times \R^2 \into M$ by
$$\gf_F(p,\lambda_1,\lambda_2) = e^{\lambda_1 X_R + \lambda_2 X_I}(p).$$
Since $(\gf_F)_*(TF\oplus \R^2) = TM$, we conclude that $\gf_F$ is a local diffeomorphism and  since the flow of the vector field $\lambda_1 X_R + \lambda_2 X_I$ preserves the foliation we conclude that all leaves in a neighbourhood of $F$ are diffeomorphic to $F$, that is, 
\begin{enumerate}
\item[a)] if a leaf $F$ is compact, the map $\gf_F$ above gives a local diffeomorphism between a neighbourhood of $F$ and $F \times \mathbb{D}^2$ for which the projection onto the open disc $\mathbb{D}^{2}\subset\RR^{2}$ is the quotient map of the foliation, and 
\item[b)] the set of points which lie in a leaf diffeomorphic to $F$ is an open set.
\end{enumerate}

Let $U \subset M$ be the open set of points which lie in a leaf diffeomorphic to $F$, let $p \in \ol{U}$ and let $\alpha:\mathbb{D}^{2n-2} \into M$ be a parametrization of the leaf through $p$, with $\mathbb{D}^{2n-2}\subset \RR^{2n-2}$ an open ball. Then
\begin{equation}
\gf: \mathbb{D}^{2n-2} \times \mathbb{D}^2 \into M, \qquad \gf(x,\lambda_1,\lambda_2) = e^{\lambda_1 X_R + \lambda_2 X_I}(\alpha(x)),
\end{equation}
is a local diffeomorphism and hence its image contains a point $q \in U$, say $q = e^{\lambda_1 X_R + \lambda_2 X_I}(\alpha(x))$. Let $F$ be the compact leaf through $q$. Then $\alpha(\mathbb{D}^{2n-2})\cap \Im(\gf_F) \neq \emptyset$, as, inverting the exponential, we get $\gf_F(q,-\lambda_1,-\lambda_2) \in \Im(\alpha)$ and hence $\gf_F$ gives a diffeomorphism between $F$ and the leaf through $p$. That is, the set $U$ above is also closed, and since $M$ is connected, $U = M$ and by property {a)} we conclude that $M$ is a fibration $M \into \Sigma$ over a compact surface.
To determine $\Sigma$, we observe that $\Omega$ is basic for this fibration  since it is closed and annihilates vertical vectors. Thererefore $\Sigma$ has a nowhere vanishing closed 1-form, giving $\Sigma = T^2$.

{\it 3.} If the twisting class vanishes, we can choose the forms $B$ and $\omega$ so that $d^H  e^{B + i\omega} =0$. Changing $\Omega$ to nearby form representing a rational class transforms the structure into a proper one. Then $\omega + \frac{1}{2i}\Omega \wedge \ol{\Omega}$ defines a symplectic form on $M$, and due to \eqref{eq:nondegeneracy condition 2}, it is  symplectic on the leaves of the distribution generated by $\Omega_R'$ and $\Omega_I'$, rendering $\pi$ a symplectic fibration.
\end{proof}

\begin{corollary}
If $M$ is a compact type 1 \gcy\ then $b_1(M) \geq 2$ and $\chi_M=0$.
\end{corollary}
\begin{proof}
For any fibration $\pi:M\to T^2$, $\pi^*:H^1(T^2) \to H^1(M)$ is injective, hence $b_1(M) \geq 2$. Also, the Euler characteristic of a fibration is the product of the Euler characteristics of the base and of the fiber.
\end{proof}

These simple obstructions tell us for example that there are no type one  generalized Calabi-Yau structures on products of spheres of dimension bigger than 1. Notice however that the manifolds $S^1\times S^3$ and $S^1 \times S^5$ do admit a type one \gcs\ with topologically trivial canonical bundle \cite{Cavalcanti:2015uua} but, by the results above, these are not generalized Calabi-Yau.

\section{Generalized Calabi--Yaus from symplectic fibrations}

Theorem \ref{theo:topology1} shows that there is a relation between proper type one \gcy s and fibrations over the 2-torus with symplectic fibers. Next we prove a partial converse to this result.

\begin{theorem}\label{theo:fibration}
Let $\pi:M \to T^2$ be a symplectic fibration over the torus. Then $M$ admits a proper type 1 \gcys, integrable \wrt\ the zero 3-form,  for which the fibers of $\pi$ are the symplectic leaves. 
\end{theorem}
\begin{proof}
We will prove the theorem by constructing explicitly a closed form determining a \gcys\ with the desired properties. Since $M\to T^2$ is a symplectic fibration, say, with generic fiber $(F,\omega_F)$, there is a 2-form $\omega \in \Omega^2(M)$ such that for all $x \in T^2$, $\iota_x^*\omega = \omega_F$ where $\iota_x:(F,\omega_F) \to M$ is the inclusion of the fiber over the point $x$. Letting $\mc{V} = \operatorname{ker}\pi_*$, the form $\omega$ gives rise to a distribution transverse to the fibers:
$${\mc{H}} = \{X\in TM |\omega(X,Y) = 0 \mbox{ for all } Y \in \mc{V}\}.$$ 
By pre-composing $\omega$ with projection to $\mc{V}$ along ${\mc{H}}$, we obtain a 2-form that equals $\omega_F$ in every fiber and annihilates the transversal distribution ${\mc{H}}$. We will  still denote this new form by $\omega$.

We pick a complex structure on the 2-torus together with a specific trivialization of the canonical bundle: $\Omega \in \Omega^{1,0}(T^2)$. In the presence of the distributions $\mc{V}$ and ${\mc{H}}$, it is convenient to split the space of forms and the exterior derivative \wrt\ the splitting $TM= \mc{V} \oplus {\mc{H}}$. Indeed, we can identify $\mc{V}^* = \mathrm{Ann}({\mc{H}})$ and ${\mc{H}}^* = \mathrm{Ann}(\mc{V})$. Further, using the complex structure on the torus, we can split $\mc{H}_\C = {\mc{H}}^{1,0}\oplus {\mc{H}}^{0,1}$ and this gives a splitting of forms:
$$\wedge^lT_\C^*M  =\oplus_{k+p+q = l}\wedge^{k;p,q}T^*_\C M, \qquad  \wedge^{k;p,q}T^*_\C M = \wedge^k\mc{V}^*_\C \tensor \wedge^p{\mc{H}}^{*1,0}\tensor  \wedge^q{\mc{H}}^{*0,1}.$$
Of course, the only values of $p$ and $q$ for which these spaces are nontrivial are $0$ and $1$. We denote the space of sections of $\wedge^{k;p,q}T^*_\C M$ by $\Omega^{k;p,q}(M)$.

Since $\mc{V}$ is an integrable distribution, $d$ splits into four components according to this decomposition of forms:
$$d:\Omega^{k;p,q} (M) \to \Omega^{k+1;p,q}(M)\oplus \Omega^{k;p+1,q}(M)\oplus \Omega^{k;p,q+1}(M)\oplus  \Omega^{k-1;\bullet,\bullet}(M)$$
and we denote the different components by
\begin{align*} 
d_\mc{V}:\Omega^{k;p,q} (M) &\to \Omega^{k+1;p,q}(M),\\
\del_{\mc{H}}:\Omega^{k;p,q} (M) &\to \Omega^{k;p+1,q}(M),\\
\delbar_{\mc{H}}:\Omega^{k;p,q} (M) &\to \Omega^{k;p,q+1}(M),\\
N:\Omega^{k;p,q} (M) &\to \Omega^{k-1,\bullet,\bullet}(M).
\end{align*}

Let $\mc{U} = \{U_\alpha:\alpha \in A\}$ be an open cover of $T^2$ over which the fiber bundle can be trivialized: $\pi^{-1}(U_\alpha)= U_\alpha \times F$. Since $\iota_x^*\omega = \omega_F$ for every $x \in T^2$ the difference $\omega -\omega_F$ vanishes along the fibers, that is,
$$\omega -\omega_F \in \Omega^{1;1,0} (M) \oplus \Omega^{1;0,1} (M) \oplus\Omega^{0;1,1} (M).$$
Therefore there are forms  $A^{0,1}_\alpha \in \Omega^{1;0,1} (M)$ and $C_\alpha \in \Omega^{0;1,1} (M)$ such that
$$\omega +A^{0,1}_\alpha+ \overline{A^{0,1}_\alpha} + C_\alpha = \omega_F$$
Since $d\omega_F =0$,  the $(2,0,1)$ component of the exterior derivative of the form above is
$$\delbar_{\mc{H}}\omega + d_{\mc{V}}A^{0,1}_\alpha =0.$$

Let $\{\psi_\alpha:\alpha \in A\}$ be a partition of unity subodinate to the cover $\mc{U}$ and consider the globally defined form $A^{0,1} = \sum (\pi^*\psi_\alpha) A^{0,1}_\alpha$. Since  $\pi^*\psi_\alpha$ is fiberwise constant, $d_{\mc{V}}\pi^*\psi_\alpha = 0$ and hence we have
\begin{equation}\label{eq:exact}
\delbar_{\mc{H}}\omega + d_{\mc{V}}A^{0,1} =\delbar_{\mc{H}}\omega + \sum_\alpha \pi^*\psi_\alpha d_{\mc{V}} A^{0,1}_\alpha= \sum_\alpha \pi^*\psi_\alpha(\delbar_{\mc{H}}\omega +  d_{\mc{V}} A^{0,1}_\alpha) =0.
\end{equation}

We claim that the form 
$$\rho = e^{i(A^{0,1}+ \omega)}\wedge \pi^*\Omega$$
endows $M$ with the desired \gcys. Indeed, by construction $A^{0,1}$ restricts to the zero form in every fiber so
$$\iota_x^*\operatorname{Re}(A^{0,1}+ \omega) = \iota_x^*\omega = \omega_F,$$
showing that $A^{0,1} + \omega$ induces the given symplectic structure on the fibers. Similarly, since both $\Omega\wedge\bar{\Omega}\wedge A^{0,1}$ and $\Omega\wedge\bar{\Omega}\wedge \bar{A^{0,1}}$ vanish we have
$$\Omega\wedge\bar{\Omega}\wedge \operatorname{Im}(A^{0,1}+ i\omega)^{n-1} =\Omega\wedge\bar{\Omega}\wedge\omega^{n-1} \neq 0,$$
as $\omega$ is symplectic on the leaves of $\pi:M\to T^2$. That is, $\rho$ is a form of the correct algebraic type and we only have to check the \gcy\ condition.

Since $d\Omega =0$, the condition $d\rho=0$ is equivalent to $d(A^{0,1}+ \omega) \wedge \Omega=0$ which we check by checking the vanishing of each of its components:
\begin{align*}
(3;1,0)& \leadsto d_{\mc{V}}\omega \wedge \Omega =0 \qquad \mbox{ as }\omega \text{ is fiberwise symplectic},\\  
(2;1,1)& \leadsto (d_{\mc{V}} A^{0,1} + \delbar_\mc{H}\omega) \wedge \Omega =0 \qquad \text{ follows from \eqref{eq:exact}}. 
\end{align*}
There are no further components, since the torus has complex dimension 1 and hence the $(k;p,q)$ component of a form vanishes if either $p$ or $q$ is greater than $1$.
\end{proof}

\begin{remark}
Given a symplectic fibration over a complex manifold, the question of whether of not the total space admits a \gcs\ was studied by Bailey in \cite{MR3150703}: the total space has a \gcs\ for which the fibers are the symplectic leaves \iff\ two obstructions vanish. One can apply that theory to prove Theorem \ref{theo:fibration} and it was with the insights from \cite{MR3150703} that we obtained the present proof. Here we opted instead for a simpler, more direct proof.   
\end{remark}

In four dimensions this gives a full converse to Theorem \ref{theo:topology1}

\begin{corollary}
A compact connected four-manifold $M$ admits a type one \gcys\ integrable \wrt\ the zero 3-form \iff\ it is orientable and it fibers over $T^2$.
\end{corollary}
\begin{proof}
Every \gcm\ is orientable and according to Theorem \ref{theo:topology1}, every type one compact \gcy\ is a fibration over a 2-torus. Conversely, every orientable surface bundle can be made into a symplectic fibration, hence, due to Theorem \ref{theo:fibration}, every surface bundle over the torus can be made into a type one \gcy\ with vanishing 3-form. 
\end{proof}

\bibliographystyle{hyperamsplain} \bibliography{references}

\end{document}